\documentclass[12pt]{article}

\usepackage{amsmath}
\usepackage{amsfonts}
\usepackage{latexsym}
\usepackage{graphicx}
\usepackage{amssymb}
\usepackage{amsthm}
\usepackage[margin=2.5cm]{geometry}
\usepackage{url}
\usepackage{color}
\usepackage{enumerate}
\usepackage[shortlabels]{enumitem} 
\usepackage[small]{caption}
\usepackage{cite}

\newtheorem{definition}{Definition}
\newtheorem{lemma}[definition]{Lemma}
\newtheorem{proposition}[definition]{Proposition}

\newtheorem{conjecture}[definition]{Conjecture}
\newtheorem{theorem}[definition]{Theorem}

\newcommand{\N}{\mathbb{N}}
\newcommand{\Z}{\mathbb{Z}}

\newcommand{\til}{\widetilde}
\newcommand{\Id}{\mathrm{Id}}
\newcommand{\Lst}{\textsc{Lst}}
\newcommand{\Fst}{\textsc{Fst}}
\newcommand{\Stab}{\mathrm{Stab}}

\newcommand{\alphabet}{\mathcal A}

\newcommand{\bu}{{\bf u}}

\newcommand{\bx}{{\bf x}}
\newcommand{\bw}{{\bf w}}
\renewcommand{\L}{\mathcal L}
\newcommand{\Fact}{\L}
\newcommand{\emptyword}{\varepsilon}
\renewcommand{\P}{\mathcal P}

\begin{document}

\title{A counterexample to a question of Hof, Knill and Simon}
\author{{\sc S\'ebastien Labb\'e\thanks{With the support of NSERC (Canada)}}\\  \\
\small LIAFA,\\ [-0.6ex]
\small Universit\'e Paris Diderot,\\ [-0.6ex]
\small Paris 7 - Case 7014,\\ [-0.6ex]
\small F-75205 Paris Cedex 13\\
\small \tt labbe@liafa.univ-paris-diderot.fr}
\date{}

\maketitle

\begin{abstract}
In this article, we give a negative answer to a question of Hof, Knill and
Simon (1995) concerning purely morphic sequences obtained from primitive
morphism containing an infinite number of palindromes. Proven for the binary
alphabet by B. Tan in 2007, we show the existence of a counterexample on the
ternary alphabet.
\end{abstract}

\section{Introduction}

A morphism of monoid $\varphi:\alphabet^*\to\alphabet^*$ is in \emph{class
$\P$} if there exists a palindrome $p$ and for every $\alpha \in \alphabet$
there exists a palindrome $q_\alpha$ such that $\varphi(\alpha)=pq_\alpha$
\cite{MR1372804}.  If $\bx$ is an infinite sequence generated by
a morphism in class $\P$, then Hof, Knill and Simon proved
that $\bx$ contains an infinite number of palindrome factors.
Such morphisms allowed them to derive some discrete Schr\"odinger
operators having purely singular continuous spectrum.
For example, on the binary alphabet $\alphabet = \{a,b\}$, the morphism $\rho : a
\mapsto bbaba, b \mapsto bba$ is in class $\P$. The morphism $\theta : a \mapsto ab,
b \mapsto ba$ is not in class $\P$, but its square $\theta^2 : a \mapsto abba, b
\mapsto baab$ is in class $\P$. Therefore, the fixed point
\[
\bx_{\rho} =
\rho(\bx_{\rho}) = bbabbabbababbabbabbababbabbabbababbabbab\cdots
\]
and the fixed point, also known as the Thue-Morse word,
\[
\bx_\theta =
\theta(\bx_\theta) = abbabaabbaababbabaababbaabbabaabbaababba\cdots
\]
are both palindromic, i.e., both contain an infinite number of palindromes.
The question of Hof, Knill and Simon that motivates this article concerns the
reciprocal. In their article, they wrote ``Clearly, we could include into
class $\P$ substitutions of the form $s(b) = q_bp$. We do not know whether all
palindromic $x_s$ arise from substitutions that are in this extended class~$\P$."
It turns out that we need to extend a little bit more the class $\P$ by
including also the conjugates. Blondin Mass\'e
(proof available in \cite[Prop 3.5]{labbe_proprietes_2008}) proved that the
primitive morphism $\tau: a \mapsto abbab, b \mapsto abb$
generates a fixed point $\bx_\tau$ having an infinite number of palindromes
such that for all $\varphi\neq\Id$ verifying $\varphi(\bx_\tau)=\bx_\tau$, then
$\varphi$ is not in the above extended class $\P$.
However $\tau$ has a conjugate in class~$\P$. Therefore, the initial question
was reformulated in \cite{labbe_proprietes_2008}.
\begin{conjecture}[Hof, Knill, Simon; Blondin Mass\'e, Labb\'e]\label{conj:hks_corrige}
Let $\bu$ be the fixed point of a primitive morphism. Then,
$\bu$ is palindromic if and only if there exists a morphism $\varphi\neq\Id$
such that $\varphi(\bu)=\bu$ and $\varphi$ has a conjugate in class $\P$.
\end{conjecture}
Conjecture~\ref{conj:hks_corrige} has an important application on another
question related to palindromic sequences with finite defect stated in
\cite{MR2566171} also considered in \cite{MR2822202} for uniformly recurrent
words and still open for fixed point of primitive and injective morphisms
\cite{bucci_vaslet_2012}.
Partial results have shown that morphisms preserving palindrome complexity must
be in class~$\P$, see \cite{MR2489283} for more details.
The conjecture was answered positively for the periodic case in
\cite{MR1964623} and for the binary alphabet in
\cite{MR2363365}. Its result is stronger than the conjecture, because it shows
that the square of the morphism has a conjugate in class $\P$.
\begin{theorem}{\rm\cite{MR2363365}}\label{thm:binary}
Let $\varphi$ be a primitive morphism on a binary alphabet. Then, the fixed
point of $\varphi$ is palindromic if and only if $\varphi^2$ has a conjugate
in class $\P$.
\end{theorem}
It was also proven independantly for uniform morphisms on the binary alphabet
in \cite{labbe_proprietes_2008} using a different approach based on overlap of
factors with a generalization to palindromes fixed by any antimorphisms. In
the uniform case for binary alphabet, it appears that conjugate morphisms are
not needed.
\begin{theorem}{\rm\cite{labbe_proprietes_2008}}\label{thm:binaryuniform}
Let $\varphi$ be a uniform primitive morphism on a binary alphabet. Then, the
fixed point of $\varphi$  is palindromic if and only if $\varphi$ or
$\widetilde{\varphi}$ is in class $\P$ or the image of any letter by
$\varphi^2$ is a palindrome (which means $\varphi^2$ is in class $\P$).
\end{theorem}

In this article, we show that a result like Theorem~\ref{thm:binary} or
Theorem~\ref{thm:binaryuniform} is not possible beyond the binary alphabet.
Indeed, we show the existence of a primitive morphism having a palindromic
fixed point such that none of its power has a conjugate in class $\P$
(Lemma~\ref{lem:gammaknoclassP}). We show even more by providing a
counterexample to Conjecture~\ref{conj:hks_corrige} (Theorem~\ref{thm:main}).
We manage to prove such a result by describing explicitely the stabilizer
of a particular infinite sequence $\bx$, i.e. the monoid of morphisms that fix
$\bx$, which is a hard problem in general. We use the properties of symmetric
words, a notion strongly related to class $\P$ morphisms.

The rest of this article consists of two sections, namely some definitions and
notations and the proof of the counterexample.

\section{Definitions and notation}

\subsection{Combinatorics on words}

We borrow from M. Lothaire \cite{MR1905123} the basic terminology about words.
In what follows, $\alphabet$ is a finite {\em alphabet} whose elements are
called  {\em letters}. A {\em word} $w$ is a finite sequence of letters
$w=w_0w_1\cdots w_{n-1}$ where $n\in\N$.
The length of $w$ is $|w|=n$ and $w_i$ denotes its $i$-th letter.
By convention the \emph{empty word} is denoted $\varepsilon$ and
its length is $0$.
The set of all finite words over $\alphabet$ is denoted by $\alphabet^*$.
Endowed with the concatenation, $\alphabet^*$ is the \emph{free monoid generated by
$\alphabet$}.
The set of right infinite words is denoted by $\alphabet^{\N}$ and the set of
biinfinite words is $\alphabet^\Z$.
Given a word $w \in \alphabet^*\cup \alphabet^{\N}$,  a \emph{factor} $f$ of $w$ is a word
$f \in \alphabet^*$ satisfying $w = xfy$ for some $x \in \alphabet^*$ and $y
\in \alphabet^*\cup \alphabet^{\N}$.
If $x=\varepsilon$ (resp. $y=\varepsilon$ ) then $f$ is called 
{\em prefix} (resp. {\em suffix}) of $w$.
The set of all factors of $w$, called the \emph{language} of $w$, is denoted
by $\Fact(w)$.
A word $w$ is {\em primitive} if
it is not a power of another word, that is if $w=u^p$ for some word $u$ and
integer~$p$ then $w=u$ and $p=1$.
Two words $u$ and $v$ are {\em conjugate} when there are words $x,y$ such that
$u=xy$ and $v=yx$.
A {\em period} of a word $w$ is an integer $p<|w|$ such that $w_i=w_{i+p}$,
for all $i< |w|-p$. We recall the result of Fine and Wilf for doubly periodic
words \cite[Theorem 8.1.4]{MR1905123}.
\begin{lemma}[Fine and Wilf]\label{lem:finewilf}
Let $w$ be a word having period $p$ and $q$. If
$|w|\geq p+q-\gcd(p,q)$, then $\gcd(p,q)$ is also a period of $w$.
\end{lemma}

\subsection{Symmetric words}

The {\em reversal} of
$w = w_0w_1 \cdots w_{n-1} \in\alphabet^n$ is the word $\til{\,w\,}=w_{n-1}
w_{n-2}\cdots w_0$.  A {\em palindrome} is a word $w$ such that $
w=\til{\,w\,}$. An infinite sequence $\bw$ is \emph{palindromic} if its
language $\L(\bw)$ contains arbitrarily long palindromes.

A word $w=w_0w_1\cdots w_{n-1}$ is \emph{symmetric} if it is the product of
two palindromes if and only if there exists an integer $0\leq a \leq n-1$ such
that
\[
w_{a-i} = w_i \quad\text{for all}\quad 0\leq i\leq n-1
\]
where the indices are taken modulo $n$. We say that the
integer $a$ is a \emph{point of symmetry} of the word~$w$.
For example, the integer $6$ is the unique point of symmetry of the word
$w=abcdcbaxyzzyx$. In the table below, we verify that
$w_{6-i} = w_i$ for all $0\leq i\leq 12$:
\[
\begin{array}{c|ccccccccccccc}
i   & 0 & 1 & 2 & 3 & 4 & 5 & 6 & 7 & 8 & 9 & 10 & 11 & 12\\
\hline
w_i & a & b & c & d & c & b & a & x & y & z & z & y & x\\
6-i & 6 & 5 & 4 & 3 & 2 & 1 & 0 & 12 & 11 & 10 & 9 & 8 & 7\\
w_{6-i} & a & b & c & d & c & b & a & x & y & z & z & y & x
\end{array}
\]
The proofs of the following lemmas are easy.
\begin{lemma}\label{lem:basicpointofsymmetry}
A finite $w$ word has a point of symmetry at $a$ if and only if $w=pq$ for some
palindromes $p$ and $q$ such that $|p|=a+1$.
\end{lemma}

\begin{lemma}\label{lem:conjugatesymmetric}
Let $w$ and $w'$ be two conjugate words, i.e. $uw = w'u$ for some word $u$.
If $w$ has a point of symmetry at $a$, then $w'$ has a point of symmetry at
$a+2|u|\mod |w|$.
\end{lemma}

\begin{lemma}\label{lem:twopoints}
If a finite word $w$ has two points of symmetry $a$ and $b$, then $gcd(b-a,|w|)$
is a period of $w$. In particular, $w$ is not primitive.
\end{lemma}

\begin{proof}
Let $\bw=\cdots ww.www\cdots = \cdots w_{n-1}.w_0w_1\cdots\in\alphabet^\Z$ be
the binfinite word such that the index $0$ of both $\bw$ and $w$ coincide.
The fact that $w$ has two points of symmetry $a$ and $b$ translates into the
fact that $\bw_{a-i} = \bw_{b-i} = \bw_i$ for all $i \in \Z$.
Then, $\bw_{i} = \bw_{a-i} = \bw_{b-(a-i)} = \bw_{i+(b-a)}$ for all $i\in\Z$.
Therefore, both $|w|$ and $b-a$ are periods of $\bw$.
We conclude by Lemma~\ref{lem:finewilf} that $gcd(b-a,|w|)$
is a period of $\bw$ and of $w$.
\end{proof}

\subsection{Morphisms}

A {\em morphism} is a function $\varphi : \alphabet^* \to \alphabet^*$
compatible with concatenation, that is, such that $\varphi (uv)= \varphi(u)
\varphi (v)$ for all $u,v \in \alphabet^*$. 
The \emph{identity morphism} on $\alphabet$ is denoted by $\Id_\alphabet$ or
simply $\Id$ when the context is clear.
A morphism $\varphi$ is \emph{primitive} if
there exists an integer $k$
such that for all $\alpha\in\alphabet$,
$\varphi^k(\alpha)$ contains each letter of $\alphabet$.
A morphism is called {\em uniform} when $|\varphi(\alpha)|=|\varphi(\beta)|$
for all letters $\alpha,\beta\in\alphabet$.
A morphism also extends in a natural way to a map over $\alphabet^\N$.
The set of \emph{sturmian morphisms} is \cite{MR1265903} the monoid generated
by
\[
\left\{
\begin{array}{ccl}
a & \mapsto & ab\\
b & \mapsto & a
\end{array},
\quad\quad
\begin{array}{ccl}
a & \mapsto & ba\\
b & \mapsto & a
\end{array},
\quad\quad
\begin{array}{ccl}
a & \mapsto & b\\
b & \mapsto & a
\end{array}
\right\}.
\]
A subset $X$ of the free monoid $\alphabet^*$ is a \emph{code} if there
exists an injective morphism $\beta: \mathcal B^*\to \alphabet^*$ such that $X
= \beta(\mathcal B)$.

Recall from Lothaire \cite{MR1905123} (Section 2.3.4)
that $\varphi$ is \emph{right conjugate} of $\varphi'$,
or that $\varphi'$ is \emph{left conjugate} of $\varphi$,
noted $\varphi\triangleleft\varphi'$, if
there exists $w \in \alphabet^*$ such that
\begin{equation}
\varphi(x)w = w\varphi'(x), \quad \textrm{for all words } x \in \alphabet^*, \label{FirstCond}
\end{equation}
or equivalently that
$\varphi(\alpha)w = w\varphi'(\alpha)$, for all letters $\alpha \in \alphabet$.
Clearly, this relation is not symmetric so that we say that two
morphisms $\varphi$ and $\varphi'$ are \emph{conjugate},
if $\varphi\triangleleft\varphi'$ or $\varphi'\triangleleft\varphi$.  It is
easy to see that conjugacy of morphisms is an equivalence relation.

A morphism is \emph{erasing} if the image of one of the letters is the
empty word.
If $\varphi$ is a nonerasing morphism, we define
$\Fst(\varphi):\alphabet\to\alphabet$ to be the
function defined by $\Fst(\varphi)(a)$ is the first letter of $\varphi(a)$.
Similarly, let $\Lst(\varphi):\alphabet\to\alphabet$ be the
function defined by $\Lst(\varphi)(a)$ is the last letter of $\varphi(a)$.
When an total order on the alphabet is clear, we represent $\Fst(\varphi)$ and
$\Lst(\varphi)$ as an ordered list.


A morphism $\varphi$ is \emph{prolongable at $a$} if there is a letter $a$
such that $\varphi(a) = aw$ where $w$ is a nonempty word $w$.
If $\varphi$ is prolongable at $a$, then
\[
\bw = a w \varphi(w) \varphi(\varphi(w)) \cdots \varphi^{n}(w) \cdots
\]
is a \emph{purely morphic word}. It is also a \emph{fixed point} of
$\varphi$, i.e. $\varphi(\bw)=\bw$.
A \emph{morphic word} is the image of a pure morphic word under a morphism.

The \emph{mirror-image} of a morphism $\varphi$, denoted by $\til{\varphi}$,
is the morphism such that $\til{\varphi}(\alpha) = \til{\varphi(\alpha)}$ for
all $\alpha \in \alphabet$.
A morphism $\varphi$ is in \emph{class $\P$} if there exists a palindrome $p$
and for every $\alpha \in \alphabet$ there exists a palindrome $q_\alpha$ such
that $\varphi(\alpha)=pq_\alpha$. The mirror-image $\til{\varphi}$ of a
morphism $\varphi$ in class~$\P$ is conjugate to $\varphi$.
If a morphism has a conjugate in class $\P$, then the image of the letters are
symmetric and they all share a common symmetry point which is quite demanding
as we will see.

\begin{lemma}
If $\varphi$ has a conjugate in class $\P$, then there exists an integer
$a\in\N$ such that for all $\alpha\in\alphabet$, $\varphi(\alpha)$ is
symmetric with a point of symmetry at $a \mod |\varphi(\alpha)|$.
\end{lemma}

\begin{proof}
Let $\varphi'$ in class $\P$ be the conjugate of $\varphi$.
Let $p$ and $q_\alpha$ be the palindromes such that
$\varphi'(\alpha)=pq_\alpha$.
Then for all letters $\alpha$, $\varphi'(\alpha)$ has a point of symmetry at
$|p|-1$ (Lemma~\ref{lem:basicpointofsymmetry}). The result follows from
Lemma~\ref{lem:conjugatesymmetric}.
\end{proof}

\subsection{Stabilizer}

The \emph{stabilizer} of a right-infinite word $\bw$ over a finite alphabet
$\alphabet$, denoted by $\Stab(\bw)$, is the monoid of morphisms
$f:\alphabet^*\to\alphabet^*$ that satisfy $f(\bw)=\bw$. Its unit element is the
identity morphism. Words that have a cyclic stabilizer are called \emph{rigid}.
See \cite{MR2543346,MR2424350} for a discussion on the subject and
\cite{MR2901382} for a recently completed proof of the rigidity of
words generated by sturmian morphisms.
\begin{proposition}{\rm\cite{MR1603835,MR2901382}}
Sturmian words generated by sturmian morphisms are rigid.
\end{proposition}

\section{The counterexample}

Let $\alphabet=\{a,b,c\}$ and $\gamma:\alphabet^*\to\alphabet^*$ be defined as
\[
\gamma:a\mapsto aca, b\mapsto cab, c\mapsto b
\]
and $\bx_\gamma$ be the fixed point of $\gamma$ starting with letter $a$:
\[
\bx_\gamma = acabacacabacabacabacacabacabacacabacabac \cdots
\]
In this section, we show that $\bx_\gamma$ is a counterexample to
Conjecture~\ref{conj:hks_corrige}. More precisely, we prove the following
result.
\begin{theorem}\label{thm:main}
The fixed point $\bx_\gamma$ of the primitive morphism $\gamma$ is
palindromic, but none of the morphisms $\varphi\in\Stab(\bx_\gamma)$ such that
$\varphi\neq\Id$, has a conjugate in class~$\P$.
\end{theorem}

\subsection{$\bx_\gamma$ is palindromic}

First, we show that the sequence $\bx_\gamma$ is palindromic.
Let $(p_k)_k$ be the sequence of finite words defined recursively by
\begin{equation}
p_{-1}=a^{-1},
\quad
p_0=b
\quad\text{and}\quad
p_{k+1}=\gamma(cap_k) \quad\text{for each}\quad k>0.
\end{equation}
The first terms are $p_0=b$, $p_1=bacacab$ and $p_2=bacacabacabacabacacab$,
all three of them being palindromes.
Also, the first letter of $p_k$ is $b$ for each $k\geq0$.
The next lemma shows that $p_k$ is a sequence of palindromes. Moreover,
$p_{k}\,aca\,p_{k-1}$ is not a palindrome only because it contains
the centered factor $cab$ or $bac$ (the rest is perfectly symmetric with
respect to the center).

\begin{lemma}\label{lem:pk}
For all $k\geq0$, the following statements are verified:
\begin{enumerate}[\rm (i)]
\item $p_{k+1} = p_{k}\, aca\, p_{k-1}\, aca\, p_k$;
\item $p_{k}$ is a palindrome;
\item $p_{k-1}\,aca\,p_{k}$ is not a palindrome.
\end{enumerate}
\end{lemma}

\begin{proof}
(i)
Notice that $p_1=p_0\cdot aca\cdot p_{-1}\cdot aca\cdot
p_0=bacaa^{-1}acab=bacacab$ and $p_2=p_1\cdot aca\cdot p_0\cdot aca\cdot p_1$.
Suppose that $p_{k}=p_{k-1}\cdot aca\cdot p_{k-2}\cdot aca\cdot p_{k-1}$, for
some $k\geq 2$.
Then,
\begin{eqnarray*}
p_{k+1}
&=& \gamma(ca p_k) \\
&=& \gamma(ca\, p_{k-1}\, aca\, p_{k-2}\, aca\, p_{k-1})\\
&=& \gamma(ca p_{k-1})\, \gamma(a)\, \gamma(ca p_{k-2})\, \gamma(a)\, \gamma(ca p_{k-1}) \\
&=& p_{k}\, aca\, p_{k-1}\, aca\, p_k.
\end{eqnarray*}

(ii) We remark that $p_0$ and $p_1$ are palindromes. Also, from (i), if
$p_{k-1}$ and $p_k$ are palindromes, then $p_{k+1}$ is a palindrome.
Therefore $p_{k}$ is a palindrome for all $k\geq 0$.

(iii)
First, $p_{-1}\,aca\,p_0=a^{-1}acab=cab$ and $p_0\,aca\,p_1=bacabacacab$ are
not palindromes.
Suppose $p_{k-1}\,aca\,p_{k}$ is not a palindrome for some $k\geq 1$.
Then $p_{k}\,aca\,p_{k+1}=p_{k}\,aca\,p_{k}\,aca\,p_{k-1}\,aca\,p_{k}$ is a
palindrome if and only if the middle part $p_{k}\,aca\,p_{k-1}$ is a
palindrome. But, it is not being the reversal of $p_{k-1}\,aca\,p_{k}$.
\end{proof}

\begin{lemma}\label{lem:infty}
$\bx_\gamma$ is palindromic.
\end{lemma}

\begin{proof}
We show that the infinite sequence $(p_k)_{k\geq0}$ shown to be palindromes at
Lemma~\ref{lem:pk} are included in $\L(\bx_\gamma)$.
Note that the first letter of $p_k$ is $b$ for each $k\geq0$.
Since each letter $b$ is preceded by the factor $ca$ in the language of
$\bx_\gamma$, we conclude that if $p_k\in\L(\bx_\gamma)$, then
$cap_k\in\L(\bx_\gamma)$ so that $p_{k+1}=\gamma(cap_k)\in\L(\bx_\gamma)$.
Therefore, $p_k\in\L(\bx_\gamma)$ for all $k\geq0$.
\end{proof}

Note that the previous two results are quite interesting. For a morphism
$\varphi:\alpha\mapsto pq_\alpha$ in class~$\P$, one shows that if $w$ is a
palindrome, then $\varphi(w)p$ is also a palindrome. Here, the result says
that prepending the non palindromic word $baca$ to $\gamma(p_k)$ makes it a
new longer palindrome. This is a first indication that $\bx_\gamma$ contains
an infinite number of palindromes for a different reason than for well known
fixed point of class $\P$ morphisms.

\subsection{The stabilizer of $\bx_\gamma$}

In this section, we describe exactly the stabilizer $\Stab(\bx_\gamma)$ of the
sequence $\bx_\gamma$. This is a hard problem in general.
Over a ternary alphabet, the monoid of morphisms generating a given infinite
word by iteration can be infinitely generated \cite{MR2424350}.
To achieve this, we show that the purely morphic sequence $\bx_\gamma$ defined
above is equal to the morphic sequence $\pi(\bw_\mu)$ where
\[
\mu\left\{
\begin{array}{l}
x \mapsto xy\\
y \mapsto xxy
\end{array}\right.,
\quad\quad
\quad\quad
\pi\left\{
\begin{array}{l}
x \mapsto ac\\
y \mapsto ab
\end{array}\right.
\]
and $\bw_\mu$ is the infinite fixed point of $\mu$ starting with letter $x$:
\[
\bw_\mu = xyxxyxyxyxxyxyxxyxyxxyxyxyxxyxyxxyxyxyxx \ldots
\]
The finite prefix of $\bx_\gamma$ shown earlier seems to indicate that the
letter $a$ appears exactly at the even positions and that $b$ and $c$ letters
appears at odd positions. This can be formulated by saying that $\bx_\gamma$
is the image by the morphism $\pi$ of some infinite word over $\{x,y\}$. The
next lemma shows that this infinite word is unique and is $\bw_\mu$.

\begin{lemma}\label{lem:corresondancexw}
We have $\gamma\circ\pi = \pi\circ\mu$.
There exists a unique infinite word $\bw$ such that
$\bx_\gamma=\pi(\bw)$. Moreover, $\bw=\bw_\mu$.
\end{lemma}

\begin{proof}
First, we verify that $\gamma\circ\pi = \pi\circ\mu$. It is sufficient to
check it for letters:
\[
\begin{array}{c}
\gamma(\pi(x))=\gamma(ac)=acab = \pi(xy)=\pi(\mu(x)),\\
\gamma(\pi(y))=\gamma(ab)=acacab = \pi(xxy)=\pi(\mu(y)).
\end{array}
\]
The word $\bx_\gamma$ is an infinite product of two kind of blocks of lengths
$2$ namely $ac$ and $ab$.
Since $\{ac, ab\}$ forms a code,
there exists a unique infinite word $\bw$ such that
$\pi(\bw)=\bx_\gamma$. But, using $\gamma\circ\pi = \pi\circ\mu$, we get
\[
\bx_\gamma = \gamma(\bx_\gamma) = \gamma(\pi(\bw)) = \pi(\mu(\bw)).
\]
From the injectivity of $\pi$, we conclude that $\bw=\mu(\bw)$ is a fixed
point of $\mu$, i.e., $\bw=\bw_\mu$.
\end{proof}

\begin{lemma}\label{lem:oddoreven}
Let $\varphi\in\Stab(\bx_\gamma)$.
Then, $|\varphi(a)|$, $|\varphi(b)|$ and $|\varphi(c)|$ are all odd or all
even.
\end{lemma}
\begin{proof}
We have that $\varphi(a)$ is a prefix of $\bx_\gamma$ and it must start with
$a$. Therefore, $\varphi(b)$ and $\varphi(c)$ both do not end with $a$.

If $|\varphi(a)|$ is odd, then $\varphi(a)$ ends with letter $a$, $\varphi(c)$
does not start with letter $a$ so that $|\varphi(c)|$ is odd. Then
$|\varphi(aca)|$ is odd and hence $\varphi(b)$ does not start with letter $a$
from what we get that $|\varphi(b)|$ is odd as well.

If $|\varphi(a)|$ is even, then $\varphi(a)$ does no end with letter $a$, $\varphi(c)$
starts with letter $a$ so that $|\varphi(c)|$ is even. Then
$|\varphi(aca)|$ is even and hence $\varphi(b)$ starts with letter $a$
from what we get that $|\varphi(b)|$ is even as well.
\end{proof}

\noindent
The morphism $\mu$ is sturmian being factorized as $\mu=(x\mapsto xy,y\mapsto
x)\circ(x\mapsto x,y\mapsto yx)$.
Therefore, the sequence $\bw_\mu$ is rigid. This is what allows us to describe
exactly the stabilizer of~$\bx_\gamma$.

\begin{proposition}\label{prop:stab}
$\Stab(\bx_\gamma) = \{\varphi :\text{there exists an integer $k$ such
that }\varphi\circ\pi = \gamma^k\circ\pi\}$.
\end{proposition}

\begin{proof}
Let $\varphi\in\Stab(\bx_\gamma)$.
From Lemma~\ref{lem:oddoreven}, $\varphi(ac)$ and $\varphi(ab)$ both have an
even length. Therefore, there exist words $u,v\in\{x,y\}^*$ such that
$\varphi(\pi(x)) = \varphi(ac) = \pi(u)$ and
$\varphi(\pi(y)) = \varphi(ab) = \pi(v)$.
Since the images of letters by $\pi$ is a code, the words $u$ and $v$ are
uniquely determined. Hence, let $\sigma:x\mapsto u,y\mapsto v$. We have that
$\varphi\circ\pi=\pi\circ\sigma$.
From Lemma~\ref{lem:corresondancexw}, $\bx_\gamma=\pi(\bw_\mu)$. We get
\[
\bx_\gamma = \varphi(\bx_\gamma) = \varphi(\pi(\bw_\mu)) = \pi(\sigma(\bw_\mu)).
\]
From the injectivity of $\pi$, we get that
$\bw_\mu=\sigma(\bw_\mu)$, that is $\sigma\in\Stab(\bw_\mu)$.
Since the fixed point $\bw_\mu$ is sturmian and rigid, more precisely
$\Stab(\bw_\mu)=\langle\mu\rangle$, there exists an integer $k$ such that
$\sigma=\mu^k$. Finally, from Lemma~\ref{lem:corresondancexw}, we have
$\pi\circ\mu^k = \gamma^k\circ\pi$.
Thus, $\varphi\circ\pi=\pi\circ\sigma = \pi\circ\mu^k = \gamma^k\circ\pi$.

Reciprocally, suppose $\varphi$ is such that $\varphi\circ\pi =
\gamma^k\circ\pi$. Then,
\[
\varphi(\bx_\gamma) = \varphi(\pi(\bw_\mu)) = \gamma^k(\pi(\bw_\mu)) =
\gamma^k(\bx_\gamma) = \bx_\gamma.
\]
Thus, $\varphi\in\Stab(\bx_\gamma)$.
\end{proof}

In other words, $\varphi\in\Stab(\bx_\gamma)$ if and only if
$\varphi(ac)=\gamma^k(ac)$ and $\varphi(ab)=\gamma^k(ab)$ for some integer
$k\geq 0$. If $k=0$, then $\varphi$ is the identity morphism. If $k=1$, then
$\varphi$ is one of the following four morphisms:
\[
\gamma_0\left\{
\begin{array}{l}
a \mapsto \emptyword\\
b \mapsto acacab\\
c \mapsto acab
\end{array}\right.,\quad
\gamma_1\left\{
\begin{array}{l}
a \mapsto a\\
b \mapsto cacab\\
c \mapsto cab
\end{array}\right.,\quad
\gamma_2\left\{
\begin{array}{l}
a \mapsto ac\\
b \mapsto acab\\
c \mapsto ab
\end{array}\right.,\quad
\gamma=\gamma_3\left\{
\begin{array}{l}
a \mapsto aca\\
b \mapsto cab\\
c \mapsto b
\end{array}\right..
\]
The indices for $\gamma$ are chosen according to the length of the image of
$a$.
If $k=2$, then $\varphi$ is one of the following eight morphisms:
\[
\gamma_0^2\left\{
\begin{array}{l}
a \mapsto \emptyword\\
b \mapsto acabacabacacab\\
c \mapsto acabacacab
\end{array}\right.,
\gamma_1^2\left\{
\begin{array}{l}
a \mapsto a\\
b \mapsto cabacabacacab\\
c \mapsto cabacacab
\end{array}\right.,
\gamma_2\gamma_1\left\{
\begin{array}{l}
a \mapsto ac\\
b \mapsto abacabacacab\\
c \mapsto abacacab
\end{array}\right.,
\]
\[
\gamma_3\gamma_1\left\{
\begin{array}{l}
a \mapsto aca\\
b \mapsto bacabacacab\\
c \mapsto bacacab
\end{array}\right.,
\gamma_2^2\left\{
\begin{array}{l}
a \mapsto acab\\
b \mapsto acabacacab\\
c \mapsto acacab
\end{array}\right.,
\gamma_1\gamma_3\left\{
\begin{array}{l}
a \mapsto acaba\\
b \mapsto cabacacab\\
c \mapsto cacab
\end{array}\right.,
\]
\[
\gamma_2\gamma_3\left\{
\begin{array}{l}
a \mapsto acabac\\
b \mapsto abacacab\\
c \mapsto acab
\end{array}\right.,
\gamma_3^2\left\{
\begin{array}{l}
a \mapsto acabaca\\
b \mapsto bacacab\\
c \mapsto cab
\end{array}\right..
\]
However, $\{\gamma_0, \gamma_1, \gamma_2, \gamma_3\}$ do not form a set of
generators for $\Stab(\bx_\gamma)$ as there are two counterexamples when $k=3$.
Both
\[
\begin{array}{l}
a\mapsto acabacac, b\mapsto abacabacacabacabacabacacab, c\mapsto abacabacabacacab,\\
a\mapsto acabacaca, b\mapsto bacabacacabacabacabacacab, c\mapsto bacabacabacacab
\end{array}
\]
satisfy $\varphi(ac)=\gamma^3(ac)$ and $\varphi(ab)=\gamma^3(ab)$ but they are
not in $\langle \gamma_0, \gamma_1\gamma_2, \gamma_3\rangle$.


\subsection{Proof of Theorem~\ref{thm:main}}

First we show that the powers of $\gamma$ itself do not have any conjugate in
class~$\P$. This implies that results like Theorem~\ref{thm:binary} or
Theorem~\ref{thm:binaryuniform} are not possible beyond the binary alphabet.

\begin{lemma}\label{lem:gammaknoclassP}
For all $k\geq1$, $\gamma^k$ does not have any conjugate in class $\P$.
\end{lemma}
\begin{proof}
For all even $k$, $\Fst(\gamma^k)=[a,b,c]$ and $\Lst(\gamma^k)=[a,b,b]$.
For all odd $k$, $\Fst(\gamma^k)=[a,c,b]$ and $\Lst(\gamma^k)=[a,b,b]$.
Therefore, $\gamma^k$ does not have conjugates other than itself. Also,
$\Fst(\gamma^k)$ contains more then one letter (similarly for
$\Lst(\gamma^k)$), so that the image of each letter by $\gamma^k$ must be a
palindrome in order to be in class $\P$. But it is not the case since
$\Fst(\gamma^k)\neq\Lst(\gamma^k)$ for all~$k$.
\end{proof}

\noindent
We show that the words $\gamma^k(ab)$ and $\gamma^k(ac)$ are symmetric with
distinct points of symmetry.
\begin{lemma}\label{lem:points_symmetry}
We have
\begin{enumerate}[\rm (i)]
\item $\gamma^k(ab) = aca\,p_{k-2}\,aca\,p_{k-1}$ is a symmetric word with
a point of symmetry at $|p_{k-2}| + 5$ for all $k\geq 2$,
\item $\gamma^k(ac) = aca\,p_{k-1}$ is a symmetric word with
a point of symmetry at $2$ for all $k\geq 1$,
\end{enumerate}
\end{lemma}

\begin{proof}
(i) We have $\gamma^2(ab)=aca\,b\,aca\,bacacab=aca\,p_0\,aca\,p_1$.  Suppose
$\gamma^k(ab) = aca\,p_{k-2}\,aca\,p_{k-1}$ for some $k\geq2$.  Then
\[
\gamma^{k+1}(ab) = \gamma(\gamma^k(ab))
= \gamma(aca\,p_{k-2}\,aca\,p_{k-1})
= \gamma(a) \gamma(cap_{k-2})\gamma(a) \gamma(cap_{k-1})
= aca\, p_{k-1}\, aca\, p_{k}.
\]

(ii)
We verify $\gamma(ac)=acab=aca\,p_0$.
Suppose $\gamma^k(ac) = aca\,p_{k-1}$ for some $k\geq 1$.
Then
\[
\gamma^{k+1}(ac) = \gamma(\gamma^k(ac))
= \gamma(aca\,p_{k-1})
= \gamma(a) \gamma(cap_{k-1})
= aca\, p_{k}.\qedhere
\]
\end{proof}

\noindent
We need also the property that the words $\gamma^k(ab)$ and
$\gamma^k(ac)$ are primitive. This is a consequence of the more general
following result showing that $\mu$ preserves primitivity as well as $\gamma$
under some more conditions.
\begin{lemma}\label{lem:primitivity}
We have
\begin{enumerate}[\rm (i)]
\item for all $w\in\{x,y\}^*$, $\mu(w)$ is primitive if and only if $w$ is primitive,
\item for all $w\in\{ab|ac\}^*$, $\gamma(w)$ is primitive if and only if $w$ is primitive,
\end{enumerate}
\end{lemma}
\begin{proof}
(i) $(\implies)$ If $w$ is not primitive, it can be written as $w=u^p$ for
some word $u$ and integer $p\geq2$. Therefore, $\mu(w)=\mu(u^p)=\mu(u)^p$ is
not primitive.
$(\impliedby)$ If $\mu(w)$ is not primitive, there exists a word $u$ and an
integer $p\geq2$ such that $\mu(w)=u^p$. Then $u$ starts with the letter $x$
and ends with letter $y$. Therefore, $u$ can be desubstituted uniquely as
$u=\mu(v)$ for some word $v$. We get $\mu(w)=u^p=\mu(v)^p=\mu(v^p)$. From the
fact that the letter images of $\mu$ forms a prefix code, we get that $w=v^p$
and $w$ is not primitive.

(ii) $(\implies)$ The proof is the same as in (i).
$(\impliedby)$ If $\gamma(w)$ is not primitive, there exists a word $u$ and an
integer $p\geq2$ such that $\gamma(w)=u^p$. The last letter of $u$ must be $b$
and the first letter of $u$ must be $a$. Hence, $u$ can be written as the
image under $\gamma$ of some word $v$. The word $v$ is unique because the
images of letters by $\gamma$ forms a prefix code. We get
$\gamma(w)=u^p=\gamma(v)^p=\gamma(v^p)$ and thus $w=v^p$ because $\gamma$
forms a code. Then, $w$ is not primitive.
\end{proof}

\noindent
We now gathered enough information to prove the main result.
\begin{proposition}\label{prop:noconjuclassP}
If $\varphi\in\Stab(\bx_\gamma)$ and $\varphi\neq\Id$, then $\varphi$ does not
have a conjugate in class $\P$.
\end{proposition}

\begin{proof}
Let $\varphi\in\Stab(\bx_\gamma)$.
From Proposition~\ref{prop:stab}, we have that $\varphi(ab)=\gamma^k(ab)$ and
$\varphi(ac)=\gamma^k(ac)$ for some integer $k$.
We may suppose that $k\geq1$ since $\varphi$ is not the identity.
Suppose by contradiction that $\varphi$ has a conjugate in class $\P$.
We may assume that $k\geq2$ since we may check that none of
$\gamma_0$, $\gamma_1$, $\gamma_2$, $\gamma_3$
defined above have a conjugate in class $\P$.
This means that there exists a word $w$ and palindromes $q$, $p_a$, $p_b$ and
$p_c$ such that
\[
\begin{array}{l}
w\,p_aq = \varphi(a)\,w,\\
w\,p_bq = \varphi(b)\,w,\\
w\,p_cq = \varphi(c)\,w,
\end{array}
\quad\quad
\text{or}
\quad\quad
\begin{array}{l}
p_aq\,w = w\,\varphi(a),\\
p_bq\,w = w\,\varphi(b),\\
p_cq\,w = w\,\varphi(c).
\end{array}
\]
Therefore $\varphi(ab)$ and $\varphi(ac)$ are symmetric word with the same
axis of symmetry. Indeed, from the equations above we get that
$\varphi(ab)$ and $\varphi(ac)$ are conjugate to symmetric words having a
point of symmetry at $|p_a|-1$ (Lemma~\ref{lem:basicpointofsymmetry}):
\[
\begin{array}{c}
\varphi(ab)\cdot w = w\cdot p_aqp_bq,\\
\varphi(ac)\cdot w = w\cdot p_aqp_cq.
\end{array}
\quad\quad
\text{or}
\quad\quad
\begin{array}{c}
w\cdot \varphi(ab) = p_aqp_bq\cdot w,\\
w\cdot \varphi(ac) = p_aqp_cq\cdot w.
\end{array}
\]
From Lemma~\ref{lem:conjugatesymmetric}, $\varphi(ab)$ and $\varphi(ac)$ both
have a point of symmetry at $A= |p_a|-1+2|w|$ or at $A=|p_a|-1-2|w|$.
From Lemma~\ref{lem:points_symmetry}, $\varphi(ab)$ already has a point of
symmetry at $|p_{k-2}|+5$. If $A\neq|p_{k-2}|+5$,
then $\varphi(ab)$ has two distinct points of symmetry.  Then, $\varphi(ab)$ is
periodic with period $g$ where $g$ is a divisor of $|\varphi(ab)|$
(Lemma~\ref{lem:twopoints}). Therefore,
$\varphi(ab)$ is not primitive which is a contradiction.
From Lemma~\ref{lem:primitivity}, $\varphi(ab)$ is
primitive because from the beginning $ab$ is a primitive word.

If $A=|p_{k-2}|+5$, then $2$ and $A$ are two distinct points of symmetry of
$\varphi(ac)$. We can get a contradiction using a primitivity argument as
above, but also using an argument based on palindromes. We have
\[
\varphi(ac)=aca\,p_{k-1}=aca\,p_{k-2}\,aca\,p_{k-3}\,aca\,p_{k-2}
\]
from Lemma~\ref{lem:points_symmetry}. A point of symmetry at $A=|p_{k-2}|+5$ implies that
$p_{k-3}acap_{k-2}$ is a palindrome which is a contradiction with
Lemma~\ref{lem:pk} (iii).  We conclude that none of the conjugate of
$\varphi$ are in class~$\P$.
\end{proof}

\begin{proof}[Proof of Theorem~\ref{thm:main}]
Follows from Lemma~\ref{lem:infty} and Proposition~\ref{prop:noconjuclassP}.
\end{proof}

\section{Conclusion}

We have seen that the fixed point $\bx_\gamma$ is not rigid. It is still open
to show that its stabilizer is not finitely generated.

The characteristic polynomial of the incidence matrix of $\gamma$ is not
irreducible as it factorizes as $(x - 1) \cdot (x^{2} - 2x - 1)$.
Note that $x^{2} - 2x - 1$ is the characteristic polynomial of the incidence
matrix of $\mu$. It is still an open question whether there exists a
counterexample to Conjecture~\ref{conj:hks_corrige} such that the
characteristic polynomial is irreducible, or in other words if there exists a
counterexample that can not be expressed as a non trivial morphic word.

Note that $\gamma_0:a\mapsto\emptyword, b\mapsto acacab, c\mapsto acab$ is
almost in class $\P$. This leads to think that the question of Hof,
Knill and Simon could be fixed once more by including erasing morphisms such
that their non erasing part has a conjugate in class $\P$. More investigations
need to be done here to find the new proper statement of the original question
of Hof, Knill and Simon.

\bibliographystyle{alpha}
\bibliography{biblio}

\end{document}